\documentclass[12pt,a4paper]{amsart}

\usepackage[T1]{fontenc}
\usepackage{lmodern}
\usepackage{amsthm,amscd}
\usepackage[headings]{fullpage}
\usepackage{amssymb,url}

\newcommand{\alg}{\operatorname{alg}}
\newcommand{\Qp}{\mathbf{Q}_p}
\newcommand{\Cp}{\mathbf{C}}
\newcommand{\Zp}{\mathbf{Z}_p}
\newcommand{\ZZ}{\mathbf{Z}}
\newcommand{\eps}{\varepsilon}
\newcommand{\Frac}{\operatorname{Frac}}
\newcommand{\wideg}{\operatorname{wideg}}
\newcommand{\card}{\operatorname{card}}
\newcommand{\val}{\operatorname{val}}
\newcommand{\OO}{\mathcal{O}}
\newcommand{\MM}{\mathfrak{m}}
\newcommand{\calR}{\mathcal{R}}
\newcommand{\calE}{\mathcal{E}}
\newcommand{\Tr}{\operatorname{Tr}}
\newcommand{\Nm}{\operatorname{N}}
\newcommand{\dfont}{\mathrm{D}}
\newcommand{\mfont}{\mathrm{M}}
\newcommand{\bigO}{\operatorname{O}}
\newcommand{\dcroc}[1]{[\![ #1 ]\!]}
\newcommand{\dpar}[1]{(\!( #1 )\!)}
\renewcommand{\geq}{\geqslant}
\renewcommand{\leq}{\leqslant} 
\renewcommand{\phi}{\varphi}

\theoremstyle{plain}
\newtheorem{theo}{Theorem}[section]
\newtheorem{coro}[theo]{Corollary}
\newtheorem{lemm}[theo]{Lemma}
\newtheorem{prop}[theo]{Proposition}
\newtheorem{conj}[theo]{Conjecture}
\newtheorem*{theoA}{Theorem A}
\newtheorem*{conjA}{Conjecture}
\newtheorem*{conjB}{Conjecture}
\newtheorem{rema}[theo]{Remark}

\begin{document}

\title{Substitution maps in the Robba ring}

\date{\today}

\author{Laurent Berger}
\address{Laurent Berger \\
UMPA de l'ENS de Lyon \\
UMR 5669 du CNRS}
\email{laurent.berger@ens-lyon.fr}
\urladdr{perso.ens-lyon.fr/laurent.berger/}

\begin{abstract}
We ask several questions about substitution maps in the Robba ring. These questions are motivated by $p$-adic Hodge theory and the theory of $p$-adic dynamical systems. We provide answers to those questions in special cases, thereby generalizing results of Kedlaya, Colmez, and others.\medskip

\noindent\textsc{R\'esum\'e.}
Nous posons plusieurs questions concernant les applications de substitution dans l'anneau de Robba. Ces questions sont motiv\'ees par la th\'eorie de Hodge $p$-adique et la th\'eorie des syst\`emes dynamiques $p$-adiques. Nous r\'epondons \`a ces questions dans des cas particuliers, ce qui g\'en\'eralise des r\'esultats de Kedlaya, Colmez, et d'autres.
\end{abstract}



\dedicatory{To Bernadette Perrin-Riou, on the occasion of her 65th birthday.}

\maketitle

\setcounter{tocdepth}{2}
\tableofcontents

\setlength{\baselineskip}{18pt}
\section*{Introduction and notation}

Let $p$ be a prime number. In this article, $K$ is a finite extension of $\Qp$, or more generally a finite totally ramified extension of $W(k)[1/p]$ where $k$ is a perfect field of characteristic $p$. Let $\OO_K$ denote the integers of $K$, let $\MM_K$ be the maximal ideal of $\OO_K$, let $k$ be the residue field of $\OO_K$, and let $\pi$ be a uniformizer of $\OO_K$. We fix a $p$-adic norm $|\cdot|$ on $K$.

In $p$-adic Hodge theory, the theory of $p$-adic differential equations, and the theory of $p$-adic dynamical systems, several rings of power series with coefficients in $K$ occur. There is $\calE^+ = \OO_K \dcroc{X} [1/\pi]$, and various completed localizations of that ring, denoted by $\calE$ (Fontaine's field), $\calE^\dagger$ (the overconvergent elements in $\calE$), $\calR^+$ (the power series converging on the $p$-adic open unit disk), and $\calR$ (the Robba ring). These rings are often endowed with a substitution map $\phi$ of the form $\phi : f(X) \mapsto f(s(X))$, where $s(X)$ is either a Frobenius lift (for example $X^p$ or $(1+X)^p-1$ or $\pi X + X^q$ where $q$ is a power of $p$, in $p$-adic Hodge theory), or a more general power series (for example the multiplication-by-$p$ map in a formal group, in the theory of $p$-adic dynamical systems).

When considering certain questions in the above domains, it is necessary to compute $(\Frac \calR)^{\phi=\mu}$ for $\mu \in K$ and for certain $s(X)$. This happens for example when considering questions of descent of morphisms for certain $\phi$-modules, or when considering $p$-adic dynamical systems on annuli. The computation of $(\Frac \calR)^{\phi=\mu}$ for $\mu \in K$ is particularly delicate: that computation is carried out (for certain $s(X) \in X \cdot \OO_K \dcroc{X}$) in lemma 3.2.4 of \cite{K00} as well as in lemma 32 of \cite{MZ}, but there are mistakes in both proofs. Those mistakes are discussed in remark 5.8 of \cite{KFM} and fixed in \S 5 of that paper (see also the errata to ibid.). 

We compute $(\Frac \calR)^{\phi=\mu}$ for all substitutions $\phi$ that are of finite height, namely those for which $s(X)$ belongs to $X \cdot \OO_K \dcroc{X}$ and is such that $\overline{s}(X) \in k \dcroc{X}$ is nonzero and belongs to $X^2 \cdot k \dcroc{X}$. In particular, we do not assume that $s(X)$ is a Frobenius lift. 

If $s'(0) \neq 0$, there exists (see for instance \cite{L94}) an element $\log_s(X) \in X \cdot \calR^+$ such that $\phi(\log_s) = s'(0) \cdot \log_s$, so that $\log_s^k \in (\calR^+)^{\phi=s'(0)^k}$ if $k \geq 1$. The following theorem (theorem \ref{fracphi}) sums up our main results.

\begin{theoA}
\label{theoA}
If $\phi$ is of finite height, then $(\Frac \calR)^{\phi=1} = K$. In addition,
\begin{enumerate}
\item $(\Frac \calR)^{\phi=s'(0)^k} = K \cdot \log_s^k$ if $s'(0) \neq 0$ and $k \in \ZZ$;
\item $(\Frac \calR)^{\phi=\mu} = \{0\}$ if $\mu \neq 1 \in K$ and if either $s'(0)=0$ or if $s'(0) \neq 0$ and $\mu$ is not of the form $s'(0)^k$ for some $k \in \ZZ$.
\end{enumerate}
\end{theoA}

We propose a conjecture concerning $(\Frac \calR)^{\phi=1}$ in a more general setting. We say that the substitution $\phi$ is overconvergent if $s(X)$ is in the ring of integers of $\calE^\dagger$ and if $\overline{s}(X) \in k \dpar{X}$ is nonzero and belongs to $X^2 \cdot k \dcroc{X}$. The following is conjecture \ref{fraconj}.

\begin{conjA}
\label{conjA}
If $\phi$ is overconvergent, then $(\Frac \calR)^{\phi=1} = K$.
\end{conjA}

Theorem A implies this conjecture when $\phi$ is of finite height. We also prove (corollary \ref{etalext}) that if $\calR'/\calR$ is a finite extension of Robba rings, and if conjecture \ref{fraconj} holds for $\calR$, then it holds for $\calR'$.

Finally, we propose a conjecture about those $h \in \calR$ such that $\phi(h)$ has a large annulus of convergence, when $\phi$ is of finite height. Let $\rho(s)$ be the largest norm of a zero of $s$ in the open unit disk. The following is conjecture \ref{foreg}, and we prove it in the cyclotomic case, namely when $s(X)=(1+X)^p-1$ (see proposition \ref{foregcyc}).

\begin{conjB}
\label{conjB}
If $\phi$ is of finite height and $h \in \calR$ is such that $\phi(h)$ is convergent on the annulus $\{ z,\, \rho(s) \leq |z| < 1\}$, then $h \in \calR^+$.
\end{conjB}

The definitions and properties of the rings that occur in this article are given in \S \ref{seclaz}. Overconvergent substitutions are introduced in \S \ref{secsubrob}, and conjecture \ref{fraconj} is discussed in \S \ref{secegvfr}. After that, we assume that $\phi$ is of finite height; these substitutions are discussed in \S \ref{subfinht}. A generalization of the classical operator $\psi$ is constructed in \S \ref{secpsi}. Theorem A is proved in \S \ref{secmainth}. The stability of conjecture \ref{fraconj} under finite extensions is proved in \S\ref{secphimod}, and conjecture \ref{foreg} is discussed in \S \ref{secforced}.

\section{Rings of power series}
\label{seclaz}

We start by defining the rings that occur in this article. There is $\OO_K \dcroc{X}$, the ring $\calE^+ = \OO_K \dcroc{X} [1/\pi]$, the ring $\OO_{\calE}$ of power series $\sum_{n \in \ZZ} a_n X^n$ with $a_n \in \OO_K$ and $a_{-n} \to 0$ as $n \to + \infty$, and the field $\calE=\OO_{\calE}[1/\pi]$. If $I$ is a subinterval of $[0;1[$, we have the ring $\calR^I$ of power series $a(X) = \sum_{n \in \ZZ} a_n X^n$ with $a_n \in K$ such that $|a_n| \cdot r^n \to 0$ as $n \to \pm \infty$ for all $r \in I$. We let $\calR^+ = \calR^{[0;1[}$ be the ring of power series $a(X) = \sum_{n \geq 0} a_n X^n$ with $a_n \in K$ such that $|a_n| \cdot r^n \to 0$ as $n \to +\infty$ for all $0 \leq r < 1$. Inside $\calR^{[r;1[}$ we have the subring $\calE^{[r;1[}$ of power series with bounded coefficients. The ring $\calE^{[r;1[}$ is contained in $\calE$. Finally, we have the field $\calE^\dagger = \cup_{r < 1} \calE^{[r;1[}$ of overconvergent elements of $\calE$, and the Robba ring $\calR = \cup_{r < 1} \calR^{[r;1[}$. We have $\calE^\dagger  \subset \calE$ and $\calE^\dagger \subset \calR$, and $\calR = \calR^+ + \calE^\dagger$ while $\calR^+ \cap \calE^\dagger = \calE^+$.

The rings $\calR^I$ are studied in \cite{L62}. We recall below the results that we use in this article; the proofs can be found in \cite{L62}. If $I = [r;s]$ is a closed interval, then $\calR^I$ is a PID. For $r<1$, the ring $\calR^{[r;1[}$ is  a Bezout domain, and $\calE^{[r;1[}$ is a PID. In particular, it makes sense to talk about the gcd of two elements of these rings, and to say that two elements are coprime. 

Let $\Cp$ be the completion of an algebraic closure $K^{\alg}$ of $K$, so that $\MM_{\Cp}$ is the $p$-adic open unit disk.  If $I$ is a subinterval of $[0;1[$, let $A(I)$ denote the annulus $A(I) = \{ z \in \Cp$ such that $|z| \in I \}$. Special cases include $D=A([0;1[)$, the open unit disk, $D(r) = A([0;r])$, the closed disk of radius $r$, and $C(r) = A( [r;r] )$, the circle of radius $r$. An element of $\calR^I$ defines a function on $A(I)$, which can have zeroes.

\begin{lemm}
\label{invnozer}
If $h \in \calR^I$, then $h$ is invertible if and only if $h$ has no zeroes in $A(I)$.
\end{lemm}

If $r \in I$, we have the norm $| \cdot |_r$ on $\calR^I$ given by $|a|_r = \sup_{n \in \ZZ} |a_n| \cdot r^n $. If $r \in |\Cp|$, then $|a|_r = \sup_{z \in \Cp,\, |z|=r} |a(z)|$. The norm $| \cdot |_r$ is multiplicative: $|ab|_r = |a|_r \cdot |b|_r$. The function $s \mapsto \log |h|_s$ is a log-convex function on $I$.

\begin{lemm}
\label{frechrp}
The family of norms $\{ |\cdot|_s\}_{r \leq s <1}$ defines a Fr\'echet structure on $\calE^{[r;1[}$, and the Fr\'echet completion of $\calE^{[r;1[}$  is $\calR^{[r;1[}$.
\end{lemm}

\begin{lemm}
\label{bdzer}
If $h \in \calR^{[r;1[}$, the following are equivalent:
\begin{enumerate}
\item $h \in \calE^{[r;1[}$;
\item $h$ has finitely many zeroes in $A([r;1[)$;
\item the function $s \mapsto |h|_s$ is bounded as $s \to 1$.
\end{enumerate}
In particular, if $h \in \calR$ and $h \notin \calE^\dagger$, then $s \mapsto |h|_s$ is eventually increasing as $s \to 1$.
\end{lemm}

\begin{coro}
\label{invrp}
We have $(\calR^+)^\times = (\calE^+)^\times$.
\end{coro}

\begin{proof}
This follows from lemmas \ref{invnozer} and \ref{bdzer}.
\end{proof}

\begin{lemm}
\label{merholom}
If $g/h \in \Frac \calR^+$ has no poles, then $g/h \in \calR^+$.
\end{lemm}

\begin{proof}
We can assume that $g$ and $h$ are coprime, so that $h$ has no zeroes. The function $h$  is then invertible in $\calR^+$ by lemma \ref{invnozer}, so that $g/h \in \calR^+$.
\end{proof}

\begin{lemm}
\label{frpdag}
We have $\Frac \calR^+ \cap \calE^\dagger = \Frac \calE^+$.
\end{lemm}

\begin{proof}
Take $g/h \in \Frac \calR^+$, and assume that $g$ and $h$ are coprime. If $g/h \in \calE^\dagger$, then $g$ and $h$ can only have finitely many zeroes, and hence both lie in $\calE^+$ by lemma \ref{bdzer}. 
\end{proof}

\begin{lemm}
\label{rpted}
If $g \in \calR$,  there exists $g^+ \in \calR^+$ and $g^\dagger \in \calE^\dagger$ such that $g=g^+ \cdot g^\dagger$.
\end{lemm}

\begin{proof}[Sketch of proof]
Take $g \in \calR^{[r;1[}$. There exists $g^+ \in \calR^+$ whose divisor (see \cite{L62}) is that of $g$, so that $g^+$ divides $g$ in $\calR^{[r;1[}$, and the quotient is in $\calE^{[r;1[}$ by lemma \ref{bdzer}. 
\end{proof}

\begin{lemm}
\label{algrob}
The field $K$ is algebraically closed inside $\Frac \calR$.
\end{lemm}

\begin{proof}
Let $F$ be a finite extension of $K$. We show that $F \otimes_K \Frac \calR \to \Frac (F \otimes_K \calR)$ is injective. If $f_1,\hdots,f_n$ is a basis of $F/K$, and if $\sum f_i \otimes a_i(X)/b_i(X) = 0$ in $\Frac (F \otimes_K \calR)$, then let $c_i(X) = \prod_{j \neq i} b_i(X)$. We have $\sum f_i \otimes a_i(X)c_i(X) = 0$ in $\Frac (F \otimes_K \calR)$ and hence in $F \otimes_K \calR$ so that $a_ic_i = 0$ for all $i$. The map is therefore injective, so that $F \otimes_K \Frac \calR$ is a domain. This would not be the case if there was a $K$-embedding of $F$ in $\Frac \calR$.
\end{proof}

\begin{rema}
\label{rtensf}
The proof of lemma \ref{algrob} shows that $\Frac (F \otimes_K \calR) = F \otimes_K \Frac \calR$.
\end{rema}

The ring $\calR^{]0;1[}$ is the ring of power series converging on the punctured open unit disk.

\begin{prop}
\label{bdrp}
If $h(X) \in \calR^{]0;1[}$ and $|h|_r$ is bounded as $r \to 0$, then $h \in \calR^+$.
\end{prop}

\begin{proof}
Write $h(X) = \sum_{k \in \ZZ} h_k X^k$. We have $|h|_r = \max_k |h_k| r^k$. If $|h|_r \leq C$ for $r$ small enough, then $|h_{-k}| \leq C r^k$ as $r \to 0$, so that $h_{-k} = 0$ if $k \geq 1$.
\end{proof}

\section{Overconvergent substitution maps}
\label{secsubrob}

If $s(X) \in \OO_{\calE}$ is such that $\overline{s}(X) \in k \dpar{X}$ is nonzero and belongs to $X \cdot k \dcroc{X}$, then $s(X)^{-1} \in \OO_{\calE}$ and $s(X)^n \to 0$ in $\OO_{\calE}$ (for the weak topology) as $n \to + \infty$, so that if $f(X) \in \calE$, then $f(s(X))$ converges in $\calE$. This way, we get a substitution map $\phi : f \mapsto f \circ s$ that generalizes the Frobenius lifts (corresponding to those $s(X)$ such that $\overline{s}(X) = X^q$ where $q$ is a power of $p$, such as $X^q$ or $(1+X)^p-1$ or $X^q + \pi X$). Analogous  maps $\phi$ are studied in $p$-adic Hodge theory, and in the theory of $p$-adic dynamical systems.

Let $\OO_{\calE}^\dagger = \OO_{\calE} \cap \calE^\dagger$. In this section, we assume that $s(X) \in \OO_{\calE}^\dagger$ and we study the restriction of $\phi$ to $\calE^\dagger$ and its extension to $\calR$.

\begin{lemm}
\label{oedag}
If $h(X) \in \OO_{\calE}^\dagger$, there exists $r_h <1$ such that $h \in \calE^{[r_h;1[}$ and $|\pi h|_r < 1$ for all $r_h \leq r < 1$.
\end{lemm}

\begin{proof}
Write $h=h^++h^-$ (according to positive and negative powers of $X$). There exists $s<1$ such that $h^- \in \calE^{[s;+\infty[}$. The function $r \mapsto |h^-|_r$ is defined for all $r \geq s$ and decreasing and $|h^-|_1 \leq 1$ since $h_n \in \OO_K$ for all $n$. Hence there exists $1 > r_h \geq s$ such that $|\pi h^-|_r < 1$ for all $r_h \leq r < 1$. Since $|\pi h^+|_r \leq |\pi|$ for all $r<1$, the claim follows.
\end{proof}

We now assume that our substitution map is given by a series $s(X) \in \OO_{\calE}^\dagger$ such that $\overline{s}(X) \in k \dpar{X}$ is nonzero and belongs to $X^2 \cdot k \dcroc{X}$. The $X$-adic valuation of $\overline{s}$ is the Weierstrass degree $\wideg(s)$ of $s$. In other words, we can write $s(X) = s^+(X) + \pi \cdot s^-(X)$ with $s^+ \in X \cdot \OO_K \dcroc{X}$ and $\wideg(s^+) = d$ for some $d \geq 2$, and $s^- \in \OO_{\calE}^\dagger$. Lemma \ref{oedag} implies that we can write $s(X)/X^d = s_d \cdot (1 + g)$ where $g \in \OO_{\calE}^\dagger$ and $s_d \in \OO_K^\times$ and there exists $r_s < 1$ such that $|g|_r < 1$ for all $r_s \leq r < 1$. 

If $h(X) = \sum_{n \in \ZZ} h_n X^n \in \calR^{[r;1[}$ for $r \geq r_s$, the series 
\begin{equation}
\label{ofs}\tag{$\Phi$}
\sum_{n \in \ZZ} h_n s_d^n X^{dn} (1+g)^n = \sum_{n \in \ZZ, \,k \geq 0}  h_n s_d^n X^{dn} \binom{n}{k} g^k
\end{equation}
converges in $\calR^{[r^{1/d};1[}$. We let $\phi(h)$ denote the sum of the series on the right. If $h \in \calE^{[r;1[}$, then $\phi(h) \in \calE^{[r^{1/d};1[} \subset \calE$ coincides with $\phi(h)$ as defined at the beginning of this section.

\begin{prop}
\label{sfo}
If $r_s<1$ is as above and if $r \geq r_s$, then 
\begin{enumerate}
\item $\phi(\calR^{[r;1[}) \subset \calR^{[r^{1/d};1[}$
\item if $|z| \geq r^{1/d}$ and $h \in \calR^{[r;1[}$, then $|s(z)| = |z|^d \geq r$, and $\phi(h)(z) = h(s(z))$
\item $|\phi(h)|_{r^{1/d}} = |h|_r$.
\end{enumerate}
\end{prop}

\begin{proof}
This is clear from equation ($\Phi$) and the definition of $r_s$.
\end{proof}

\section{Eigenvalues of $\phi$ and $(\Frac \calR)^{\phi=1}$}
\label{secegvfr}

In \S \ref{secmainth} below, we prove that $(\Frac \calR)^{\phi=1} =K$ if we assume that $s(X) \in X \cdot \OO_K\dcroc{X}$ (theorem \ref{fracphi}). We expect this result to hold for a general overconvergent substitution.

\begin{conj}
\label{fraconj}
We have $(\Frac \calR)^{\phi=1} =K$ in general.
\end{conj}

In the rest of this section, we give some results related to this conjecture. These results are not used in the rest of the article. We say that $\lambda \in \calR$ is an eigenvalue of $\phi$ is there exists a nonzero $h \in \calR$ such that $\phi(h)= \lambda \cdot h$. This terminology is not quite correct as $\phi$ is only a semilinear map on $\calR$.

\begin{prop}
\label{conjequiv}
We have $(\Frac \calR)^{\phi=1} =K$ iff $\dim_K \calR^{\phi=\lambda} \leq 1$ for all $\lambda \in \calR$.
\end{prop}

\begin{proof}
If $g,h \in \calR^{\phi=\lambda}$, then $g/h \in (\Frac \calR)^{\phi=1}$. If $(\Frac \calR)^{\phi=1} = K$, then $g/h \in K$ so that $\dim_K \calR^{\phi=\lambda} \leq 1$. Conversely, take $g/h \in (\Frac \calR)^{\phi=1}$, so that $g \cdot \phi(h) = h \cdot \phi(g)$. Take $r<1$ such that $g,h,\phi(g)$ and $\phi(h)$ belong to $\calR^{[r;1[}$. We can assume  that $g$ and $h$ are coprime in $\calR^{[r;1[}$, and then $g$ divides $\phi(g)$ and $h$ divides $\phi(h)$ in $\calR^{[r;1[}$. The common quotient $\lambda \in \calR$ is such that $g,h \in \calR^{\phi=\lambda}$. If $\dim_K \calR^{\phi=\lambda} = 1$, then $g/h \in K$.
\end{proof}

\begin{prop}
\label{egvoed}
If $\lambda \in \calR$ and $\calR^{\phi = \lambda} \neq \{ 0 \}$, then $\lambda \in \OO_{\calE}^\dagger$.
\end{prop}

\begin{proof}
Take $h \in \calR^{\phi = \lambda}$. If $h \in  \calE^\dagger$, we can assume that $h \in \OO_{\calE}^\times$ and then $\lambda = \phi(h)/h \in \OO_{\calE} \cap \calE^\dagger = \OO_{\calE}^\dagger$. If $h \notin \calE^\dagger$, then $r \mapsto |h|_r$ is eventually increasing as $r \to 1$ by lemma \ref{bdzer}. If $r<1$ is close enough to $1$, then $|\phi(h)|_r = |h|_{r^d}$ by proposition \ref{sfo}, so that $|\lambda|_r = |h|_{r^d} / |h|_r \leq 1$. By lemma \ref{bdzer}, $\lambda \in \calE^\dagger$. In addition, if we write $\lambda(X) = \sum \lambda_n X^n$, then $|\lambda_n| r^n \leq 1$ for all $n$ and all $r<1$ close to $1$, hence $\lambda_n \in \OO_K$ for all $n$.
\end{proof}

\begin{prop}
\label{robinvar}
We have $\calR^{\phi=1} = K$.
\end{prop}

\begin{proof}
Take $g \in \calR$ such that $\phi(g)=g$. Proposition \ref{sfo} implies that $|g|_{s^{1/d}}=|g|_s$ if $s$ is close to $1$, so that the function $s \mapsto |g|_s$ is bounded as $s \to 1$. By lemma \ref{bdzer}, we have $g \in \calE^\dagger$. Take $r \geq r_s$ so that by proposition \ref{sfo}, $\phi(\calE^{[r;1[}) \subset \calE^{[r^{1/d};1[}$ and if $|y| \geq r^{1/d}$ and $h \in \calE^{[r;1[}$, then $|s(y)| \geq r$, and $\phi(h)(y) = h(s(y))$. If $|z| \geq r$ and $s(y)=z$, then $|y| \geq r^{1/d}$. Therefore if $g(z)=0$, then $g(y) = \phi(g)(y) = g(z) = 0$. This implies that if $g$ has a zero in $A([r;1[)$, then it has infinitely many zeroes. By lemma \ref{bdzer}, this is not possible if $g \in \calE^\dagger$.

Pick $z \in A([r;1[) \cap K^{\alg}$. The function $g \mapsto g(z)$, from $(\calE^{[r;1[})^{\phi=1}$ to $K(z)$, is therefore injective, and hence $(\calE^{[r;1[})^{\phi=1}$ is a finite dimensional $K$-vector space. It is also a domain, and hence a field extension of $K$. Since $K$ is algebraically closed in $\calE^\dagger$, we get that $(\calE^{[r;1[})^{\phi=1}=K$. This is true for all $r$ close to $1$, so that $(\calE^\dagger)^{\phi=1} = K$.
\end{proof}

We say that $s(X)$ is a $p$-power lift if in $k \dcroc{X}$, we have $\overline{s}(X) \in k \dcroc{X^p}$. This is equivalent to saying that $s'(X) \in \pi \OO_{\calE}^\dagger$. Frobenius lifts are examples of $p$-power lifts.

\begin{prop}
\label{pizdim}
If $s(X)$ is a $p$-power lift, and $\lambda \in \OO_{\calE}^\dagger$, then $\dim_K \calR^{\phi=\lambda} \leq n(\lambda)-1$ where $n(\lambda) \in \ZZ$ is such that $n(\lambda) > 2 \cdot \val(\lambda) / \val(s') +1$.
\end{prop}

\begin{proof}
Take $f \in \calR^{\phi=\lambda}$, so that $f(s(X)) = \lambda(X) \cdot f(X)$. We have \[ f'(s(X)) \cdot s'(X) = \lambda(X) \cdot f'(X) + \lambda'(X) \cdot f(X) \] and hence $\phi(f') \in f' \cdot \lambda/s' + \calR \cdot f$. Likewise for all $m \geq 1$, \[ \phi(f^{(m)}) \in f^{(m)} \cdot \lambda/(s')^m + \calR \cdot f + \calR \cdot f' + \cdots + \calR \cdot f^{(m-1)}. \]
Given $n$ elements $f_1,\hdots,f_n$ of $\calR$, let $W(f_1,\hdots,f_n)$ denote their Wronskian 
\[ W(f_1,\hdots,f_n) = \det 
\begin{pmatrix} 
f_1 & \cdots & f_n \\
f_1' & \cdots & f_n' \\
\vdots & & \vdots \\
f_1^{(n-1)} & \cdots & f_n^{(n-1)} 
\end{pmatrix} \] 
If $f_1,\hdots,f_n \in \calR^{\phi=\lambda}$, then $W(f_1,\hdots,f_n)$ belongs to $\calR^{\phi = \lambda^n / (s')^{n(n-1)/2}}$ by the above. Take $n=n(\lambda) \in \ZZ$ such that $n > 2 \cdot \val(\lambda) / \val(s') +1$. By proposition \ref{egvoed}, $\calR^{\phi = \lambda^n / (s')^{n(n-1)/2}} = \{0\}$, and hence $W(f_1,\hdots,f_n) = 0$, so that $f_1,\hdots,f_n$ are linearly dependent over $K$. 

Therefore, $\dim_K \calR^{\phi=\lambda} \leq n(\lambda)-1$.
\end{proof}

\begin{rema}
\label{rphinvoc}
If $s(X)$ is a $p$-power lift,  and $\lambda=1$, we can take $n(\lambda) = 2$, and we get a new proof that $\calR^{\phi=1}=K$ in this case.
\end{rema}

\begin{prop}
\label{dimuniv}
If there exists $C \in \ZZ_{\geq 1}$ such that $\dim_K \calR^{\phi=\lambda} \leq C$ for all $\lambda \in \OO_{\calE}^\dagger$, then $\dim_K \calR^{\phi=\lambda} \leq 1$ for all $\lambda \in \OO_{\calE}^\dagger$.
\end{prop}

\begin{proof}
Take $g,h \in \calR^{\phi=\lambda}$, and $m \geq 1$. The $m+1$ functions $\{ g^i h^{m-i} \}_{0 \leq i \leq m}$ all belong to $\calR^{\phi=\lambda^m}$. If $m \geq C$, they are linearly dependent over $K$. Hence $g/h$ is algebraic over $K$ in $\Frac \calR$. Therefore, $g/h \in K$ by lemma \ref{algrob}.
\end{proof}

We finish this section with some additional motivation for conjecture \ref{fraconj}. Suppose that $s(X) \in X \cdot \OO_K \dcroc{X}$ with $s'(0) \neq 0$, and that $u(X) \in X \cdot \OO_K \dcroc{X}$ is such that $u \circ s = s \circ u$. Let $\tilde{u} \in \calR^+$ be the Lie logarithm of $u$, as defined in \S 4 of \cite{L94}. We have (lemma 4.4.2 of ibid.) $\tilde{u} \circ s = s' \cdot \tilde{u}$. Hence $\tilde{u}$ is an eigenvector of $\phi$ for the eigenvalue $s'$. If $v(X)  \in X \cdot \OO_K \dcroc{X}$ is another series such that $v \circ s = s \circ v$, then $\tilde{u}$ and $\tilde{v}$ are both eigenvectors of $\phi$ for the eigenvalue $s'$, and conjecture \ref{fraconj} (which holds, by theorem \ref{fracphi}, if $s(X) \in X \cdot \OO_K \dcroc{X}$) along with proposition \ref{conjequiv} then implies that $\tilde{v} = c \cdot \tilde{u}$ for some $c \in K$. We can  use this to show that $u$ and $v$ commute with each other for composition. If $s(X) \in X \cdot \OO_K \dcroc{X}$ with $s'(0) \neq 0$, there is a much simpler proof of this, but knowing conjecture \ref{fraconj} in greater generality would allow us to prove similar results for $p$-adic dynamical systems on annuli, where they are currently not known.

\section{Substitutions of finite height}
\label{subfinht}

In this section (and in the rest of this article), we assume that $s(X)$ is of finite height, namely that it belongs to $\OO_K \dcroc{X}$, and that $s(0)=0$. Recall that $\wideg(s) = d$ is finite and that $d \geq 2$. In other words, $s(X) = \sum_{k \geq 1} s_k X^k$ with $s_1,\hdots,s_{d-1} \in \MM_K$ and $s_d \in \OO_K^\times$. 

\begin{rema}
\label{zercent}
If $s(X)\in \OO_K \dcroc{X}$ and $\wideg(s) = d \geq 2$,  there exists $a \in \MM_K$ such that $s(a) = a $, so that $s(X)$ is conjugate to the power series $s_a(X) = s(X+a)-a$ which is such that $\wideg(s_a) = d$ and $s_a(X)=0$. Hence the condition that $s(0)=0$ can be achieved by a simple conjugation. 
\end{rema}

\begin{proof}
Let $\val(\cdot) = -  \log_p |\cdot|$. Since $\wideg(s) \geq 2$, the Newton polygon of $s(X)-X$ starts with a segment of length $1$ and slope $-\val(s(0))$, which gives us such an $a$ with $\val(a)=\val(s(0))$.
\end{proof}

Recall that if $r < 1$, $D(r)$ is the closed disk of radius $r$ and $C(r)$ is the circle of radius $r$. Define a function $\lambda : [0;1] \to [0;1]$ by $\lambda(r) = \max_k |s_k| r^k$. 

\begin{lemm}
\label{lambehv}
We have $\lambda(r) = r^d$ if $r$ is close enough to $1$, $\lambda(r) < r$ for all $r>0$, and $\lambda(r) \leq |\pi|r$ if $r$ is close enough to $0$.
\end{lemm}

\begin{proof}
These all follow easily from the formula $\lambda(r) = \max_k |s_k| r^k$. 
\end{proof}

\begin{prop}
\label{fsurj}
If $r<1$, then $s(D(r)) = D(\lambda(r))$ and $s(C(r))$ contains $C(\lambda(r))$.
\end{prop}

\begin{proof}
That $s(D(r)) \subset D(\lambda(r))$ follows from the definition $\lambda(r) = \max_k |s_k| r^k$. For the second assertion, it is better to use valuations. Let $\val(\cdot) = -  \log_p |\cdot|$. Define $\lambda^*(v) = \min_k \val(s_k) + k v$. If $\val(z) = \lambda^*(v)$, choose an index $j$ such that $\val(z)= \val(s_j) + jv$. The line with equation $y=\val(s_j) +v \cdot (j-x)$ passes through $(0,\val(z))$, lies below the Newton polygon of $s$, and touches it at the point $(j,\val(s_j))$. Hence the equation $s(X)-z$ has a root of valuation $v$.
\end{proof}

\begin{coro}
\label{supfrob}
If $h(X) \in \calR^+$, then $|\phi(h)|_r = |h|_{\lambda(r)}$ for all $r<1$.
\end{coro}

\begin{prop}
\label{rpinv}
We have $(\calR^+)^{\phi=1} = K$.
\end{prop}

\begin{proof}
Take $g \in \calR^+$ such that $\phi(g)=g$. We have $g(z) = g(s(z))$ for all $z \in D$. Since $s^{\circ n}(z) \to 0$ as $n \to +\infty$, we have $g(z)=g(0)$ for all $z \in D$ and hence $g \in K$.
\end{proof}

\begin{prop}
\label{frpinvmu}
If $\mu \in K$, and $f \in (\Frac \calR^+)^{\phi=\mu}$, then $f^{\pm 1} \in \calR^+$.
\end{prop}

\begin{proof}
Write $f=g/h \in (\Frac \calR^+)^{\phi=\mu}$. Let $z$ be a nonzero zero (or pole) of $g/h$. We have $\mu \cdot (g/h)(z) = \phi(g/h)(z) = (g/h)(s(z))$ so that $s(z)$ is itself a zero (or pole) of $g/h$. Likewise $s^{\circ n}(z)$ is a zero (or pole) of $g/h$ for all $n \geq 0$. Since $|s(x)| < |x|$ if $x \neq 0$, and since the zeroes and poles of $g/h$ cannot accumulate towards $0$, we must have $s^{\circ n}(z) = 0$ for $n \gg 0$. Therefore $0$ is a zero (or pole) of $g/h$. Consequently, either $0$ is a zero of $g/h$ and $g/h$ only has zeroes, or $0$ is a pole of $g/h$ and $g/h$ only has poles. By lemma \ref{merholom}, either $g/h$ or $h/g$ belongs to $\calR^+$. 
\end{proof}

\begin{coro}
\label{frpinv}
We have $(\Frac \calR^+)^{\phi=1} = K$.
\end{coro}

\begin{proof}
This follows from propositions \ref{frpinvmu} and \ref{rpinv}.
\end{proof}

\begin{rema}
\label{subrzo}
We have $\phi(\calR^+) \subset \calR^+$ and $\phi(\calR^{[r;1[}) \subset \calR^{[r^{1/d};1[}$ if $r \geq r_s$.
\begin{enumerate}
\item If the only zero of $s$ in $D$ is $0$, then $\phi(\calR^{[r;1[}) \subset \calR^{[r^{1/d};1[}$ for all $0 \leq r <1$. 
\item It is not true in general that $\phi$ preserves $\calR^{]0;1[}$. For example, $1/X \in \calR^{]0;1[}$ but $\phi(1/X) = 1/s(X)$, and that series belongs to $\calR^{[r;1[}$ only if $r$ is larger than the norm of all the zeroes of $s(X)$. See \S \ref{secforced} for a precise conjecture regarding this.
\end{enumerate}
\end{rema}

\begin{proof}
We prove (1). An element of $\calR^{[r;1[}$ is the sum of an element of $\calR^+$ and of $\sum_{n \geq 1} h_n / X^n$ where $|h_n|r^{-n} \to 0$. We have $s(X)=s_d X^d \cdot u(X)$ with $u(X) \in 1+X \OO_K \dcroc{X}$ and $s_d \in \OO_K^\times$. The claim now follows since $1/s(X)^n = 1/X^{nd} \cdot u(X)^{-n}$ and $u(X)^{-n} \in 1+X \OO_K \dcroc{X}$, and since $\phi(\calR^+) \subset \calR^+$. 
\end{proof}

\begin{prop}
\label{prodphi}
If $a(X) \in \calE^+$ and $a(0)=1$, the product $\prod_{i=0}^\infty a(s^{\circ i}(X))$ converges in $\calR^+$ to an element $m_a(X) \in \calR^+$ such that $\phi(m_a) \cdot a = m_a$.

If in addition $a(X) \in 1+X \cdot \OO_K \dcroc{X}$, then $m_a(X) \in 1+X \cdot \OO_K \dcroc{X}$ as well.
\end{prop}

\begin{proof}
The second claim follows from the first, since $a(s^{\circ i}(X)) \in 1+X \cdot \OO_K \dcroc{X}$ for all $i$ in this case. The first claim follows from lemma \ref{frechrp} and the fact that for a given $r<1$, we have $|a(s^{\circ i}(X))-1|_r \to 0$ as $i \to + \infty$.
\end{proof}

\begin{rema}
\label{k45}
Compare with remark 4.5 of \cite{KFM}.
\end{rema}

\section{The operator $\psi$}
\label{secpsi}

We have $\OO_{\calE}/ \pi \OO_{\calE} = k \dpar{X}$. Since $\wideg(s) = d$, $k \dpar{X}$ is a free $k \dpar{s(X)}$-vector space of dimension $d$. Hence $\OO_{\calE}$ is a free $\phi(\OO_{\calE})$-module of rank $d$, and we get a ``trace'' map $\psi : \OO_{\calE} \to \OO_{\calE}$ defined on $\OO_{\calE}$ by $\phi(\psi(h)) = \Tr_{\OO_{\calE} / \phi(\OO_{\calE})} h(X)$. This map extends to $\calE$. We have $\psi(1) = d$ and $\psi(f \cdot \phi(g)) = \psi(f) \cdot g$. 

\begin{rema}
\label{psivsf}
In $p$-adic Hodge theory, the operator $\psi$ is usually defined as either $1/d$ or $1/\pi$ times our $\psi$ defined above. See for instance \S I.2 of \cite{CSP}.
\end{rema}

In the rest of this section, we assume that $s(X)$ is of finite height. The ring $k \dcroc{X}$ is a free $k \dcroc{s(X)}$-module of rank $d$. Hence $\OO_K \dcroc{X}$ is a free $\phi(\OO_K \dcroc{X})$-module of rank $d$, and if $h(X) \in \OO_K \dcroc{X}$, then $\phi(\psi(h)) = \Tr_{\OO_K \dcroc{X} / \phi(\OO_K \dcroc{X})} h(X)$. This shows that $\psi$ preserves $\calE^+$. We next study the restriction of $\psi$ to $\calE^\dag$.

\begin{lemm}
\label{pclem}
If $0< r <1$ and $n \geq 1$, then $\psi(1/X^n) \in \calE^{[\lambda(r);1[}$, and $|\psi(1/X^n)|_{\lambda(r)} \leq r^{1-n}/\lambda(r)$.
\end{lemm}

\begin{proof}
Let $q(X) = s(X)/X = s_1 + s_2 X + \cdots \in \OO_K \dcroc{X}$. The formula
\[ \psi \left(\frac{1}{X^n}\right) = \psi \left(\frac{q(X)^n }{s(X)^n}\right) = \frac{\psi(q(X)^n) }{X^n} \]
shows that $\psi(1/X^n) \in \calE^+[1/X] \subset \calE^{[\lambda(r);1[}$. We now prove the bound on $|\psi(1/X^n)|_{\lambda(r)}$. 

By corollary \ref{supfrob}, we have $\max_k |s_k| r^k = |s|_r = |X|_{\lambda(r)}$. This implies that $|s_j / X|_{\lambda(r)} \leq r^{-j}$ for all $j \geq 1$. In addition, $|s/X \cdot \psi(X^k)|_{\lambda(r)} \leq 1/\lambda(r)$ for all $s \in \OO_K$ and $k \geq 0$.

We can write 
\[ \psi \left(\frac{1}{X^n}\right) =  \psi \left(\frac{q(X) }{s(X) X^{n-1}}\right) = \frac{1}{X} \psi \left(\frac{1}{X^{n-1}}(s_1+s_2X+\cdots)\right) = \sum_{j \geq 1} \frac{s_j}{X} \psi \left(\frac{1}{X^{n-j}}\right). \]

If $n=1$, this formula and the above observations imply that $|\psi(1/X)|_{\lambda(r)} \leq 1/\lambda(r)$. If $n \geq 2$ and $1 \leq j \leq n-1$, then by induction, we get
\[ \left|\frac{s_j}{X}  \psi\left(\frac{1}{X^{n-j}}\right)\right|_{\lambda(r)} \leq r^{-j} \cdot r^{1-n+j}/\lambda(r) \leq r^{1-n}/\lambda(r). \]
If $j \geq n$, then $|s_j/X \cdot \psi(X^{j-n})|_{\lambda(r)} \leq 1/\lambda(r) \leq r^{1-n}/\lambda(r)$. This implies the claim.
\end{proof}

\begin{rema}
\label{fxc}
Slightly different estimates for certain $r$ are proved in lemma I.9 of \cite{CSP} for $s(X)=(1+X)^p-1$ and in proposition 2.2 of \cite{FX} for $s(X)=\pi \cdot X + X^q$.
\end{rema}

\begin{prop}
\label{psiconv}
We have $\psi(\calE^{[r;1[}) \subset \calE^{[\lambda(r);1[}$ for $r<1$.
\end{prop}

\begin{proof}
An element of $\calE^{[r;1[}$ is the sum of an element of $\calE^+$ and of $\sum_{n \geq 1} h_n / X^n$ where $|h_n|r^{-n} \to 0$. We have $\psi(\calE^+) \subset \calE^+$, and $| \psi(h_n/X^n) |_{\lambda(r)} \to 0$ as $n \to + \infty$ by lemma \ref{pclem}. The claim follows.
\end{proof}

\begin{rema}
\label{phibet}
Since $\psi(\phi(h)) = d \cdot h$, we recover, when $s(X)$ is a Frobenius lift, the (unproved) corollary 5.3 of \cite{KFM}. 
\end{rema}

We now show that $\psi$ extends from $\calE^+$ to $\calR^+$. Recall (lemma \ref{frechrp}) that the family of norms $\{ |\cdot|_r\}_{r<1}$ defines a Fr\'echet structure on $\calE^+$, and that the completion of $\calE^+$ is $\calR^+$.

\begin{prop}
\label{psirp}
The map $\psi : \calE^+ \to \calE^+$ is uniformly continuous for the family of norms $\{ |\cdot|_r\}_{r<1}$, and extends to a map $\psi : \calR^+ \to \calR^+$.
\end{prop}

\begin{proof}
We have $\OO_K \dcroc{X} = \oplus_{j=0}^{d-1} \OO_K \dcroc{s(X)} \cdot X^j$. If $h(X) \in \calE^+$, we can therefore write it as $h(X) = \sum_{i \geq 0} \sum_{j=0}^{d-1} h_{i,j} s(X)^i X^j$ with $\{ h_{i,j} \}$ a bounded sequence of $K$. 

By lemma \ref{lambehv}, there exists $r_0<1$ such that if $r_0 \leq r <1$, then $|s|_r = r^d$. In this case, $|h|_r = \max_{i,j} |h_{i,j}| r^{di+j}$. We then have $\psi(h) = \sum_{j=0}^{d-1} \psi(X^j) \sum_{i \geq 0} h_{i,j} X^i$. This implies that if $r \geq r_0$, there exists a constant $C(r)$ such that $|\psi(h)|_{r^d} \leq C(r) \cdot |h|_r$. The map $\psi$ is therefore uniformly continuous, and extends from $\calE^+$ to $\calR^+$.
\end{proof}

Since $\psi$ is defined on $\calE^\dagger$ and on $\calR^+$, it extends to $\psi : \calR \to \calR$, and we have $\psi(\calR^{[r;1[}) \subset \calR^{[\lambda(r);1[}$ by proposition \ref{psiconv}. We finish this section with a few results that are not used in the rest of the paper.

\begin{prop}
\label{dualpsi}
Let $e_1,\hdots,e_d$ be a basis of $\OO_K \dcroc{X}$ over $\OO_K \dcroc{s(X)}$.

There exists $\delta(X) \neq 0 \in \calE^+$ and $e_1^*,\hdots,e_d^* \in \delta(X)^{-1} \cdot \calE^+$ such that $\psi(e_i^* e_j) = \delta_{ij}$.
\end{prop}

\begin{proof}
Let $\delta(X) = \det (\Tr_{\calE^+ / \phi(\calE^+)} (e_i e_j))_{i,j} \in \calE^+$. The set $e_1,\hdots,e_d$ is a basis of $k \dpar{X}$ over $k \dpar{s(X)}$, and hence of $\calE$ over $\phi(\calE)$, so that $\Tr_{\calE^+ / \phi(\calE^+)} (e_i e_j) = \Tr_{\calE / \phi(\calE)} (e_i e_j)$. The field extension $\calE / \phi(\calE)$ is separable, hence $\delta(X) \neq 0$. We have $e_i^* = \sum_k \phi(g_{i,k}) e_k$ where $(g_{i,k})_{i,k} = ( \psi (e_m e_n)_{m,n})^{-1}$, so that $e_i^* \in \delta(X)^{-1} \cdot \calE^+$.
\end{proof}

\begin{coro}
\label{rpsumfr}
We have $\calR^+ = \oplus_{i=1}^d \phi(\calR^+) \cdot e_i$.
\end{coro}

\begin{proof}
We have $\calE^+ = \oplus_{i=1}^d \phi(\calE^+) \cdot e_i$, and if $h = \sum \phi(h_i) \cdot e_i$, then $h_i = \psi(he_i^*)$. All the underlying maps extend by uniform continuity to $\calR^+$.
\end{proof}

\begin{rema}
\label{discpsi}
In the cyclotomic and Lubin-Tate cases, $\delta(X) \in (\calE^+)^\times$. However, if $s(X)=X^d$, then $\delta(X)$ is a multiple of $X^{d(d-1)}$. In general, the discriminant $\delta(X)$ is equal to $\Nm_{\calE^+ / \phi(\calE^+)} s'(X)$ since $\calE = \phi(\calE)[X]$.
\end{rema}

\begin{rema}
\label{nosumrr}
Corollary \ref{rpsumfr} cannot be pushed too far. For example, if $s'(0) \neq 0$ (which holds in the cyclotomic and Lubin-Tate cases), then $K \dcroc{X} = K \dcroc{s(X)}$.
\end{rema}

\section{The space $(\Frac \calR)^{\phi=\mu}$}
\label{secmainth}

In this section, we prove theorems A and B. Recall that $s(X)$ is of finite height.

\begin{prop}
\label{scalegv}
If $\mu \in K$ and $h \in \calR^{\phi=\mu}$, then $h \in \calR^+$.
\end{prop}

\begin{proof}
By applying $\psi$ to $\phi(h)=\mu \cdot h$, we get $\psi(h) = d/\mu \cdot h$. Repeatedly applying proposition \ref{psiconv} shows that $h \in \calR^{]0;1[}$. If $g \in \calR$, write $g=g^- + g^+$ with $g^- \in 1/X \cdot K \dcroc{1/X}$ and $g^+ \in \calR^+$. We have $\psi(h)^- = d/\mu \cdot h^-$. Lemma \ref{pclem} implies that there exists a constant $C$, depending only on $\mu/d$ and $K$, such that
\[ |h^-|_{\lambda(r)} = |\mu/d \cdot \psi(h)^-|_{\lambda(r)} \leq C \cdot r/\lambda(r) \cdot |h^-|_r. \] Iterating this gives $|h^-|_{\lambda^{\circ k}(r)}  \leq C^k \cdot r/\lambda^{\circ k}(r) \cdot |h^-|_r$. If $r$ is small enough, then $\lambda(r) \leq |\pi| r$ by lemma \ref{lambehv}. Fix such an $r$. If $n \geq 1$, then
\[ |X^n h^-|_{\lambda^{\circ k}(r)}  \leq C^k \cdot r/\lambda^{\circ k}(r) \cdot \lambda^{\circ k}(r)^n/r^n  \cdot |X^n h^-|_r \leq (C |\pi|^{n-1})^k \cdot |X^n h^-|_r. \]
If $n \geq 1$ is large enough so that $C |\pi|^{n-1} \leq 1$, proposition \ref{bdrp} implies that $X^n h^- \in \calR^+$. Hence if $\psi(h) = d/\mu \cdot h$, then $h(X) \in X^{-n} \cdot \calR^+$ for some $n \geq 0$.

If in addition $\phi(h) = \mu \cdot h$, then $h(s(X)) \in s(X)^{-n} \cdot \calR^+$. If $h$ has a pole at $0$, then it has poles at the zeroes of $s$. So unless $h \in \calR^+$, the only zeroes of $s$ are at $0$, and $0$ is then a zero of order $d$ of $s$. In this case, if $h$ has a pole of order $n$ at $0$, then $\phi(h)$ has a pole of order $dn$. We therefore have $h \in \calR^+$.
\end{proof}

\begin{rema}
\label{robinvrem}
This gives us another proof that $\calR^{\phi=1} = K$ (proposition \ref{robinvar}).
\end{rema}

\begin{proof}
If $h \in \calR^{\phi=1}$, then $h \in \calR^+$ by proposition \ref{scalegv}, and therefore $h \in K$ by proposition \ref{rpinv}.
\end{proof}

If $s'(0) \neq 0$, there exists an element $\log_s(X) \in X \cdot \calR^+$ such that $\phi(\log_s) = s'(0) \cdot \log_s$ (see for instance proposition 2.2 of \cite{L94}; if $r(X)=s(X)/(s'(0) \cdot X)$, and $m_r$ is as in proposition \ref{prodphi}, then $\log_s(X) = X \cdot m_r(X)$). Therefore $\log_s^k \in (\calR^+)^{\phi=s'(0)^k}$ if $k \geq 1$. 

\begin{theo}
\label{logegv}
If $\mu \neq 1 \in K$ and $h \neq 0 \in \calR^{\phi=\mu}$, then $s'(0) \neq 0$, and there exists $k \geq 1$ such that $\mu=s'(0)^k$ and $h \in K \cdot \log_s^k$.
\end{theo}

\begin{proof}
If there exists $\mu \neq 1 \in K$ and $h \in \calR$ such that $\phi(h) = \mu h$, then $h \in \calR^+$ by proposition \ref{scalegv}, and $h(0) = \mu h(0)$ so that $h(0)=0$. If $h(X) = h_k X^k + \bigO(X^{k+1})$ and $s(X) = s_j X^j + \bigO(X^{j+1})$, with $h_k,s_j \neq 0$, then the order of vanishing at $0$ of $\mu h$ is $k$ and that of $\phi(h)$ is $jk$, so that $j=1$. This shows that $s'(0) \neq 0$. In this case, $\phi(h) = \mu h$ implies that $\mu = s_1^k = s'(0)^k$ where $s_1,k$ are as above. Corollary \ref{frpinv} now implies that $h = c \cdot \log_s^k$ with $c \in K$.
\end{proof}

\begin{prop}
\label{fracrp}
Take $a, b \in \calE^+$ such that $a(0), b(0) \neq 0$. 

If $h \in \calR$ is such that $\phi(h)/h = a/b$, then $h \in \Frac \calR^+$.
\end{prop}

\begin{proof}
We can replace $a$ by $a/a(0)$ and $b$ by $b/a(0)$ so that $a(0)=1$. Let $m_a \in \calR^+$ be as in proposition \ref{prodphi}, so that $\phi(m_a) \cdot a=m_a$. We have $\phi(hm_a) / (h m_a) = 1/b$, so we only need to prove the claim when $a=1$.

Assume therefore that $\phi(h) = h/b$. Recall that $\rho(s)$ is the largest norm of a zero of $s$ in the open unit disk. Fix $r > \rho(s)$ such that $h \in \calR^{[r;1[}$. If $b$ has no zero in $A([r;1[)$, then $h/b \in \calR^{[r;1[}$ and $\phi(h/b) = (h/b) / \phi(b)$. If $y$ is a zero of $\phi(b)$, then $z = s(y)$ is a zero of $b$, and we have $|y| \geq \min(|z|^{1/d},|z|/|\pi|)$. We can therefore keep doing this until $\phi^{\circ n} (b)$ has a zero in $A([r;1[)$. So assume that $\phi(h) = h/b$ and that $b$ has a zero in $A([r;1[)$. Let $c$ be a full isoclinic factor of $b$ whose zeroes are in $A([r;1[)$ and such that $c(0)=1$. We have $\phi(h) \cdot b = h$ so that $c$ divides $h$ in $\calR^+$. If $h(z) = 0$ and $s(y)=z$, then $\phi(h)(y)=0$. Since $|y| > |z|$ and $c$ is isoclinic, we get that $\phi(c)$ divides $h$. By iterating this, we get that, if $m_c$ is the element attached to $c$ by proposition \ref{prodphi}, then $m_c$ divides $h$ in $\calR^+$. We then have $\phi(h/m_c) \cdot b/c = h/m_c$. This way, we can get rid of all the factors of $b$ corresponding to zeroes in $A([r;1[)$. 

By iterating the above two steps, we eventually get that $\phi(h) \cdot b = h$ where $b$ has no zeroes in $D$. Indeed, let $N(b) \subset \, ]0;1[$ denote the set of all the norms of the zeroes of $b$ (recall that $b(0) \neq 0$). Each time we divide $b$ by a full isoclinic factor, $\card N(b)$ strictly decreases. And each time we replace $b$ by $\phi(b)$, the elements of $N(b)$ strictly increase. Past the bound $r_s$ (see proposition \ref{sfo}), we have that if $|z| \geq r_s$ and $s(y)=z$, then $|y|=|z|^{1/d}$. Therefore, past that point, replacing $b$ by $\phi(b)$ will not increase $\card N(b)$, while dividing $b$ by a full isoclinic factor will strictly decrease $\card N(b)$. Hence eventually $\card N(b) = 0$.

The resulting element $b$ is therefore of the form $b(0) \cdot c$ where $c \in 1+ X \OO_K \dcroc{X}$. Applying proposition \ref{prodphi} to $c$, we get $m_c \in 1+ X \OO_K \dcroc{X}$ such that $\phi(h/m_c) \cdot b(0) = h/m_c$. Proposition \ref{scalegv} now implies that $h/m_c \in \calR^+$, and we are done.
\end{proof}

\begin{rema}
\label{fraked}
If in addition $h \in \calE^\dagger$, then $h \in \Frac \calR^+ \cap \calE^\dagger = \Frac \calE^+$ by lemma \ref{frpdag}.  

Now compare proposition \ref{fracrp} with lemma 5.4 of \cite{KFM}.
\end{rema}

\begin{theo}
\label{fracphimu}
If $\mu \in K$ and $f \in (\Frac \calR)^{\phi=\mu}$, then $f^{\pm 1} \in \calR^+$.
\end{theo}

\begin{proof}
Take $f/g \in (\Frac \calR)^{\phi=\mu}$. By lemma \ref{rpted}, we can assume that $g=g^+ \in  \calR^+$ and that $f=f^+ h$ with $f^+ \in \calR^+$ and $h \in \calE^\dagger$. We get
\[ \frac{\phi(h)}{h} =  \mu \cdot \frac{f^+ \cdot \phi(g^+)}{\phi(f^+) \cdot g^+} \in \Frac \calR^+ \cap \calE^\dagger = \Frac \calE^+\]
where the last equality follows from lemma \ref{frpdag}. Hence we can write $\phi(h)/h=a/b$ with $a,b \in \calE^+$. In addition, we can divide $f^+$ and $g^+$ by powers of $X$, and assume that $f^+(0),g^+(0) \neq 0$, and then that $a(0),b(0) \neq 0$.

By proposition \ref{fracrp}, $h \in \Frac \calR^+$. Therefore, $f/g  = f^+h/g^+$ belongs to $\Frac \calR^+$. The claim now follows from proposition \ref{frpinvmu}.
\end{proof}

\begin{rema}
\label{frdimp}
\begin{enumerate}
\item 
Compare with lemma 5.6 of \cite{KFM} (or rather its corrected version, see the errata to ibid.)
\item 
In the cyclotomic case, namely when $s(X)=(1+X)^p-1$, the computations of \S 3.2 of \cite{CLA} give a different proof of the fact that $(\Frac \calR)^{\phi=\mu} = (\Frac \calR^+)^{\phi=\mu}$.
\end{enumerate}
\end{rema}

We can now state theorem A.

\begin{theo}
\label{fracphi}
If $\phi$ is of finite height, then $(\Frac \calR)^{\phi=1} = K$. In addition,
\begin{enumerate}
\item $(\Frac \calR)^{\phi=s'(0)^k} = K \cdot \log_s^k$ if $s'(0) \neq 0$ and $k \in \ZZ$;
\item $(\Frac \calR)^{\phi=\mu} = \{0\}$ if $\mu \neq 1 \in K$ and if either $s'(0)=0$ or if $s'(0) \neq 0$ and $\mu$ is not of the form $s'(0)^k$ for some $k \in \ZZ$.
\end{enumerate}
\end{theo}

\begin{proof}
This follows from theorem \ref{fracphimu}, proposition \ref{rpinv} and theorem \ref{logegv}.
\end{proof}

\section{Application to $\phi$-modules}
\label{secphimod}

In this section, we assume that $(\Frac \calR)^{\phi=1}=K$, and we give some applications to $\phi$-modules.  A $\phi$-module over $\Frac \calR$ is a finite dimensional $\Frac \calR$-vector space, with a semi-linear map $\phi : \mfont \to \mfont$.

\begin{prop}
\label{mphinv}
If $\mfont$ is a $\phi$-module over $\Frac \calR$, then $\mfont^{\phi=1} \otimes_K {\Frac \calR} \to \mfont$ is injective. In particular, $\mfont^{\phi=1}$ is a finite dimensional $K$-vector space.
\end{prop}

\begin{proof}
Let $m_1 \otimes f_1 + \cdots + m_r \otimes f_r$ be in the kernel of the map, with $r$ minimal. We can assume that $f_1=1$. Applying $\phi$ and subtracting gives a shorter relation, which is zero by minimality, so that $\phi(f_i)= f_i$ for all $i$. Hence $f_i \in (\Frac \calR)^{\phi=1}=K$.
\end{proof}

A $\phi$-module over $\calR$ is a free $\calR$-module $\dfont$ of finite rank, with a semi-linear map $\phi : \dfont \to \dfont$ (one usually assumes in addition that $\phi^*(\dfont)=\dfont$, but we do not use this). 

\begin{coro}
\label{dphinv}
If $\dfont$ is a $\phi$-module over $\calR$, then $\dfont^{\phi=1} \otimes_K \calR \to \dfont$ is injective. In particular, $\dfont^{\phi=1}$ is a finite dimensional $K$-vector space.
\end{coro}

\begin{proof}
This follows from proposition \ref{mphinv} applied to $\mfont = \Frac \calR \otimes_{\calR} \dfont$.
\end{proof}

\begin{rema}
\label{fxproof}
This gives a proof of the unproved assertion ``Note that $\dfont^{\phi_q=1}$ is finite-dimensional over $L$'' on page 2571 of \cite{FX}.
\end{rema}

We say that $\calR'/\calR$ is a finite extension of Robba rings if $\calR'$ itself is a Robba ring with coefficients in a finite extension $L$ of $K$, and in a variable $Y$, and if $\calR'$ is a free $\calR$-module of finite rank. We also assume that $\phi$ extends to $\calR'$. These objects occur for instance in $p$-adic Hodge theory, when $\calR$ is attached to a $p$-adic field $F$, and we are given a finite extension $F'/F$. In this case there is a corresponding finite extension $\calR'/\calR$ of Robba rings as defined above (see for instance \S I.2 of \cite{LB08}). 
For example, take  $K=L=\Qp$ and $s(X)=(1+X)^p-1$ (the cyclotomic case). If $Y = X^{1/n}$ with $n \geq 1$, then $\calR'/\calR$ is a finite extension of Robba rings of degree $n$, and if $p \nmid n$ we can set
\[ \phi(Y) = \left((1+X)^p-1\right)^{1/n} = Y^p \cdot \left(1+\frac{p}{Y^n} + \cdots + \frac{p}{Y^{n(p-1)}} \right)^{1/n} \in (\calE')^\dagger.\]

\begin{coro}
\label{etalext}
Let $\calR'/\calR$ be a finite extension of Robba rings, with coefficients in $L$ and $K$, such that $\phi$ extends to $\calR'$. If $(\Frac \calR)^{\phi=1}=K$, then $(\Frac \calR')^{\phi=1}=L$.
\end{coro}

\begin{proof}
The hypotheses on $\calR'/\calR$ imply that $\Frac \calR'$ is a finite extension of $\Frac \calR$, and therefore also a $\phi$-module over $\Frac \calR$. By proposition \ref{mphinv}, $(\Frac \calR')^{\phi=1}$ is a finite dimensional $K$-vector space. It is also a field extension of $L$. The corollary now results from lemma \ref{algrob} applied to $\calR'$.
\end{proof}

\section{Convergence close to the origin}
\label{secforced}

We still assume $s(X)$ to be of finite height. Recall (see remark \ref{subrzo}) that it is not true in general that $\phi$ preserves $\calR^{]0;1[}$. For example, $1/X \in \calR^{]0;1[}$ but $\phi(1/X) = 1/s(X)$, that belongs to $\calR^{[r;1[}$ only if $r> \rho(s)$. We propose the following conjecture.

\begin{conj}
\label{foreg}
If $h \in \calR$ is such that $\phi(h) \in \calR^{[\rho(s);1[}$, then $h \in \calR^+$.
\end{conj}

\begin{prop}
\label{foregcyc}
Conjecture \ref{foreg} is true in the cyclotomic case, namely when $s(X)=(1+X)^p-1$.
\end{prop}

\begin{lemm}
\label{transconv}
Let $S$ be the set of sequences $\{x_k\}_{k \geq 0}$ with $x_k \in K$. Define an operator $T : S \to S$ by the formula $(Tx)_\ell = \sum_{k=0}^\ell (-1)^k \binom{\ell}{k} x_k$. If both sequences $x$ and $Tx$ converge to $0$, then $x=0$.
\end{lemm}

\begin{proof}
Suppose that $x$ converges to $0$, and let $f : \Zp \to K$ be given by the formula $f(z) = \sum_{k \geq 0} (-1)^k \binom{z}{k} x_k$. The function $f$ is continuous and $(Tx)_\ell = f(\ell)$. If $f(\ell) \to 0$ as $\ell \to +\infty$, then $f=0$ by continuity, and hence $x=0$.
\end{proof}

\begin{proof}[Proof of proposition \ref{foregcyc}]
Let $\eps$ be a primitive $p$-th root of $1$. Since $s(X)=(1+X)^p-1$, we have $\rho(s)=\rho=|\eps-1|$. Take $r > \rho$ and $g(X) = \sum_{n \geq 1} g_n / X^n \in \calE^{[r;1[}$. We have $g((1+X)\eps-1) = g(X \eps + \eps-1)  \in \calE^{[r;1[}$ and we  expand it as follows:
\[ \sum_{n \geq 1} \frac{g_n}{(X \eps + \eps-1)^n}  = \sum_{n \geq 1} \frac{g_n}{X^n} \eps^{-n} \left(1+\frac{\eps-1}{X \eps}\right)^{-n}  = \sum_{n \geq 1} \frac{g_n}{X^n} \eps^{-n} \sum_{j \geq 0} \binom{-n}{j} \left( \frac{\eps-1}{X \eps}\right)^j \]
By setting $m=n+j$ and using the fact that $\binom{-n}{j} = (-1)^j \binom{n+j-1}{j}$, we get $g((1+X)\eps-1) = \sum_{m \geq 1} b_m / X^m$ where $b_m = \eps^{-m} \sum_{n = 1}^{m} (-1)^{n-m} \binom{m-1}{n-1} g_n (\eps-1)^{n-m}$. This  gives us an explicit  formula for the coefficients of $g((1+X)\eps-1) \in \calE^{[r;1[}$. 

We now prove that if $g(X) \in \calR^{[\rho;1[}$ is such that $g((1+X)\eps-1) \in \calR^{[\rho;1[}$, then $g(X) \in \calR^+$. It is enough to prove that the negative part $\sum_{n \geq 1} g_n / X^n$ of $g$ is zero, so we assume that $g(X) = \sum_{n \geq 1} g_n / X^n$ as above. If we let $x_k = g_{k+1} (\eps-1)^{k+1}$ and $y_\ell = (-1)^\ell \eps^{\ell+1} b_{\ell+1} (\eps-1)^{\ell+1}$ for $k,\ell \geq 0$, then $y_\ell = \sum_{k = 0}^{\ell} (-1)^k \binom{\ell}{k} x_k$. The fact that $g(X) \in \calE^{[\rho;1[}$ is equivalent to $x_k \to 0$ as $k \to +\infty$, and likewise the fact that $g((1+X)\eps-1) \in \calE^{[\rho;1[}$ is equivalent to $y_\ell \to 0$ as $\ell \to +\infty$. The claim now results from lemma \ref{transconv}, applied to $\{x_k\}_{k \geq 0}$, since $y=Tx$ in the notation of that lemma.

We now prove the proposition. If $g(X) = \phi(h)(X) \in \calR^{[\rho;1[}$, then $\phi(h)(X) = \phi(h)((1+X)\eps-1)$ so that by the above claim $\phi(h)(X) \in \calR^+$. Therefore $h = 1/p \cdot \psi \phi(h) \in \calR^+$.
\end{proof}

\begin{rema}
\label{amfr}
The method of proof of proposition \ref{foregcyc} is reminiscent of the Amice-Fresnel theorem (see \cite{AF72}, Th\'eor\`eme 1 or \cite{AR}, \S 4.4 of chapter 6).
\end{rema}

\end{document}